\documentclass[a4paper,10pt, draft]{amsart}
\usepackage{amsmath}
\usepackage{amssymb}
\usepackage{amsthm}
\usepackage{comment}
\usepackage{enumerate}
\newtheorem{theorem}{Theorem}
\newtheorem*{theorem*}{Theorem}

\newtheorem*{lemma*}{Lemma}
\newtheorem{corollary}{Corollary}
\newtheorem*{corollary*}{Corollary}
\newtheorem{proposition}{Proposition}
\newtheorem*{proposition*}{Proposition}

\theoremstyle{definition}
\newtheorem{definition}{Definition}
\newtheorem*{definition*}{Definition}
\newtheorem{remark}{Remark}
\newtheorem*{remark*}{Remark}
\newtheorem{example}{Example}
\newtheorem*{example*}{Example}

\allowdisplaybreaks[1]

\begin{document}
\title[MZVs and iterated log-sine integrals]{multiple zeta values and iterated log-sine integrals}
\author{RYOTA UMEZAWA}
\date{}
\maketitle
\begin{abstract}
We introduce an iterated integral version of (generalized) log-sine integrals (iterated log-sine integrals) and prove a relation between a multiple polylogarithm and iterated log-sine integrals. We also give a new method for obtaining relations among multiple zeta values, which uses iterated log-sine integrals, and give alternative proofs of several known results related to multiple zeta values and log-sine integrals.
\end{abstract}
\section{Introduction}
Multiple zeta values (MZVs) are defined by
\[\zeta(k_{1},\dots,k_{n}) = \sum_{0 < m_1  < \dots < m_n} \frac{1}{m_{1}^{k_{1}} \cdots m_{n}^{k_{n}}},\]
where $k_{1},\dots,k_{n}$ are positive integers with $k_{n} \ge 2$. For $n=1$, these are called Riemann zeta values.

On the other hand, (generalized) log-sine integrals are defined by
\[{\rm Ls}_{k}^{(l)}(\sigma)=-\int_{0}^{\sigma}\theta^{l}\log^{k-1-l}\left|2\sin\frac{\theta}{2}\right|\,d\theta,\]
where $k\in\mathbb{N}$ and $l \in \mathbb{Z}_{\ge 0}$. For $l=0$, these are called (basic) log-sine integrals.
There are studies on relations between Riemann zeta values and (generalized) log-sine integrals, for example \cite{CCS}, \cite{L} and \cite{ZW}. Moreover, there are also studies on relations between multiple zeta values and (generalized) log-sine integrals, for example \cite{BS}. 

In this paper, we give a new method for obtaining relations among multiple zeta values using log-sine integrals (an outline is given in Section \ref{se:overview}). For this purpose, we need to introduce an iterated integral version of (generalized) log-sine integrals (iterated log-sine integrals):
\begin{definition}[Iterated log-sine integrals]
Let $A(\theta) = \log\left|2\sin(\theta/2)\right|$.
For $\sigma \in \mathbb{R}_{\ge 0}$, $\mathbf{k} = (k_{1}, \dots, k_{n}) \in \mathbb{N}^{n}$ and $\mathbf{l} = (l_{1},\dots,l_{n}) \in \mathbb{Z}_{\ge 0}^{n}$, we define
\[{\rm Ls}_{\mathbf{k}}^{\mathbf{l}}(\sigma)=(-1)^{n}\int_{0<\theta_{1}<\dots<\theta_{n}<\sigma}\prod_{u=1}^{n}\theta_u^{l_u}A^{k_u-1-l_u}(\theta_{u})\,d\theta_u.\]
\end{definition}
This integral converges absolutely when $k_{u} \ge 1$ and $l_{u} \ge 0$ for any $u \in \{1,\dots,n\}$.
In Section \ref{se:Pre}, we prove several properties related to iterated log-sine integrals. In particular, Theorem \ref{pr:prod} and Theorem \ref{th:LitoLs} in Section \ref{se:Pre} play an important role in our method.
In Section \ref{se:zeta}, we give some examples, a result obtained by using computer and alternative proofs of several known results.
For example, by using only our method, we can give another proof of Euler's result:
 \[\zeta(2k)=(-1)^{k+1}\frac{(2\pi)^{2k}B_{2k}}{2(2k)!}.\]
In Section \ref{se:algo}, we give our algorithm for computer calculations.
\section{Outline of a new method}\label{se:overview}
In this section, we give an outline of a new method for obtaining relations among multiple zeta values. 
For $\mathbf{k}=(k_{1},\dots,k_{n}) \in \mathbb{N}^n$, $|\mathbf{k}|=k_{1}+\dots+k_{n}$ is called the weight of $\mathbf{k}$ and $n$ is called the depth of $\mathbf{k}$.
Let $\{k\}^{m}$ denotes $m$ repetitions of $k$. When $k_{n}>1$, we write $\mathbf{k}$ in the form
\begin{align*}
\mathbf{k}=(\{1\}^{a_{1}-1},b_{1}+1,\{1\}^{a_{2}-1},b_{2}+1,\dots,\{1\}^{a_{h}-1},b_{h}+1),
\end{align*}
then its dual index $\mathbf{k}^{*}$ is defined by
\[\mathbf{k}^{*}=(\{1\}^{b_{h}-1},a_{h}+1,\dots,\{1\}^{b_{2}-1},a_{2}+1,\{1\}^{b_{1}-1},a_{1}+1).\]
We define $\mathbf{k}^{(0)}=\mathbf{k}$,
\begin{align*}
\mathbf{k}^{(1)}=&
  \begin{cases}
    (k_{1},\dots,k_{n-1},k_{n}-1)&\quad {\rm if}\ k_{n}>1,\\
    (k_{1},\dots,k_{n-1})&\quad {\rm if}\ k_{n}=1,\\
    \phi&\quad {\rm if}\ n=1, k_{n}=1,
  \end{cases}
\end{align*}
and $\mathbf{k}^{(n)}=(\mathbf{k}^{(n-1)})^{(1)}\ (n>1)$.

Then the result by J. M. Borwein, D. J. Broadhurst and J. Kamnitzer \cite[theorem 4.4]{BBK} can be rewritten as follows.
\begin{theorem}[{Borwein, Broadhurst and Kamnitzer \cite[theorem 4.4]{BBK}
}]\label{th:BBK}
For $\mathbf{k}=(k_{1},\dots,k_{n}) \in \mathbb{N}^{n}$ with $k_{n}>1$, we have
\begin{align}\zeta(\mathbf{k})=\sum_{m=0}^{|\mathbf{k}|}\mathrm{Li}_{\mathbf{k}^{(m)}}(e^{\frac{\pi}{3}i})\overline{\mathrm{Li}_{(\mathbf{k}^{*})^{(|\mathbf{k}|-m)}}(e^{\frac{\pi}{3}i})}.\label{eq:BBK}
\end{align}
\end{theorem}
Here, $\mathrm{Li}_{\mathbf{k}}(z)$ is the multiple polylogarithm defined by
\[\mathrm{Li}_{\mathbf{k}}(z) = \sum_{0<m_{1} < \dots < m_{n}} \frac{z^{m_{n}}}{m_{1}^{k_{1}} \cdots m_{n}^{k_{n}}} \quad \mathrm{and}\quad \mathrm{Li}_{\phi}(z) =1\]
for $|z| < 1$, and can be continued holomorphically to  $\mathbb{C}\setminus[1,\infty)$.
Theorem \ref{th:BBK} was proved as a duality among multiple Clausen values. However , we use it as a relation between a multiple zeta value and iterated log-sine integrals.

In Section \ref{se:Pre}, we prove the following two facts, which play an important role on our method.
The first is the following relation between a multiple polylogarithm and iterated log-sine integrals as a corollary of Theorem \ref{th:LitoLs}.
\begin{corollary}\label{th:LitoLspi/3}For $\mathbf{k} = (k_{1},\dots,k_{n}) \in \mathbb{N}^{n}$, we have
\begin{align*}
\mathrm{Li}_{\mathbf{k}}(e^{\frac{\pi}{3}i})=i^n\int_{0<\theta_{1}<\dots<\theta_{n}<\theta_{n+1} = \frac{\pi}{3}}\prod_{u=1}^{n} \frac{\left(A(\theta_{u+1}) - A(\theta_{u}) - \frac{i\theta_{u+1}}{2} + \frac{i\theta_{u}}{2}\right)^{k_{u}-1}}{(k_{u}-1)!}\,d\theta_{u},
\end{align*}
where $A(\theta) = \log\left|2\sin(\theta/2)\right|$.
\end{corollary}
By this corollary, we can see that multiple polylogarithms at $e^{\pi i/3}$ can be written as a $\mathbb{Q}(i)$-linear combination of iterated log-sine integrals at $\pi/3$ satisfying $k_{j} -1 - l_{j} \ge 0$ for all $ j \in \{1,\dots,n\}$. The second is the following corollary as a corollary of Theorem \ref{pr:prod}.
\begin{corollary}\label{pr:prodpi/3}
The product of iterated log-sine integrals at $\pi/3$ satisfying $k_{j} -1 - l_{j} \ge 0$ for all $ j \in \{1,\dots,n\}$ can be written as a $\mathbb{Q}$-linear combination of the products of $\pi^m\ (m \ge 0)$ and an iterated log-sine integral at $\pi/3$ satisfying $k_{j} -1 - l_{j} > 0$ for all $j \in \{1,\dots,n\}$.
\end{corollary}
By Corollary \ref{th:LitoLspi/3} and Corollary \ref{pr:prodpi/3}, the right hand side of (\ref{eq:BBK}) can be written as a $\mathbb{Q}(i)$-linear combination of the products of $\pi^m\ (m \ge 0)$ and an iterated log-sine integral at $\pi/3$ satisfying $k_{j} -1 - l_{j} > 0$ for all $j \in \{1,\dots,n\}$. After writing the right hand side of (\ref{eq:BBK}) in terms of iterated log-sine integrals, by taking the real part of it, we can obtain the iterated log-sine integral expression of multiple zeta values. In addition, by taking the imaginary part, we can obtain relations among iterated log-sine integrals. By applying these relations to the iterated log-sine integral expression of multiple zeta values, we obtain relations among multiple zeta values. This is the new method introduced in this paper. Concrete examples are given in Section \ref{se:zeta}.
\section{Iterated log-sine integrals}\label{se:Pre}
In this section, we prove several propositions on iterated log-sine integrals.
There is two fundamental propositions on iterated log-sine integrals.
The first proposition is as follows.
\begin{proposition}\label{pr:Lsshuffle}
Let $S_{n}$ be the symmetric group of degree $n$. For $\mathbf{k} = (k_{1}, \dots, k_{n})$, $\mathbf{l} = (l_{1}, \dots, l_{n})$,  $\mathbf{k}' =(k_{n+1}, \dots, k_{n+n'})$, $\mathbf{l}' =(l_{n+1}, \dots, l_{n+n'})$ and $\tau \in S_{n+n'}$,
we define
\begin{align*}\tau(\mathbf{k},\mathbf{k}') &=\tau(k_{1},\dots,k_{n+n'})= (k_{\tau({1})},\dots,k_{\tau({n+n'})}),\\
\tau(\mathbf{l},\mathbf{l}') &=\tau(l_{1},\dots,l_{n+n'})= (l_{\tau({1})},\dots,l_{\tau({n+n'})}),
\end{align*}
and 
\[S_{n,n'} = \{\tau \in S_{n+n'} \mid \tau^{-1}(1)<\dots<\tau^{-1}(n)\ and \ \tau^{-1}(n+1)<\dots<\tau^{-1}(n+n')\}.\]
Then we have
\begin{align}{\rm Ls}_{\mathbf{k}}^{\mathbf{l}}(\sigma)\cdot {\rm Ls}_{\mathbf{k}'}^{\mathbf{l}'}(\sigma) = \sum_{\tau \in S_{n, n'}}{\rm Ls}_{\tau(\mathbf{k},\mathbf{k}')}^{\tau(\mathbf{l},\mathbf{l}')}(\sigma).\label{eq:shuffle}
\end{align}
In particular, the product of two iterated log-sine integrals at $\sigma$ can be written as a $\mathbb{Q}$-linear combination of iterated log-sine integrals at $\sigma$.
Moreover this decomposition satisfies the associative law.
\end{proposition}
\begin{proof}
We have
\begin{align*}
&{\rm Ls}_{\mathbf{k}}^{\mathbf{l}}(\sigma)\cdot {\rm Ls}_{\mathbf{k}'}^{\mathbf{l}'}(\sigma) \\
&=(-1)^{n+n'}\int_{\substack{0<\theta_{1}<\dots<\theta_{n}<\sigma \\ 0<\theta_{n+1}<\dots<\theta_{n+n'}<\sigma }}\prod_{u=1}^{n+n' }\theta_{u}^{l_{u}}A^{k_{u}-1-l_{u}}(\theta_{u})\,d\theta_{u}.
\end{align*}
By separating the domain of integration as in the case of the shuffle product of multiple zeta values, we obtain the desired formula.

Let $\mathbf{k}'' = (k_{n+n'+1}, \dots, k_{n+n'+n''})$ and $\mathbf{l}'' = (l_{n+n'+1}, \dots, l_{n+n'+n''})$. Then we obtain that
\[{\rm Ls}_{\mathbf{k}}^{\mathbf{l}}(\sigma) \left( {\rm Ls}_{\mathbf{k}'}^{\mathbf{l}'}(\sigma)\cdot {\rm Ls}_{\mathbf{k}''}^{\mathbf{l}''}(\sigma)\right)\]
and 
\[\left({\rm Ls}_{\mathbf{k}}^{\mathbf{l}}(\sigma) \cdot {\rm Ls}_{\mathbf{k}'}^{\mathbf{l}'}(\sigma) \right) {\rm Ls}_{\mathbf{k}''}^{\mathbf{l}''}(\sigma)\]
are equal to 
\begin{align*}
&(-1)^{n+n'+n''}\int_{\substack{0<\theta_{1}<\dots<\theta_{n}<\sigma \\ 0<\theta_{n+1}<\dots<\theta_{n+n'}<\sigma \\ 0<\theta_{n+n'+1}<\dots<\theta_{n+n'+n''}<\sigma}}\prod_{u=1}^{n+n'+n'' }\theta_{u}^{l_{u}}A^{k_{u}-1-l_{u}}(\theta_{u})\,d\theta_{u}\\
&=\sum_{\tau \in S_{n, n', n''}}{\rm Ls}_{\tau(\mathbf{k},\mathbf{k}', \mathbf{k}'')}^{\tau(\mathbf{l},\mathbf{l}',\mathbf{l}'')}(\sigma),
\end{align*}
where 
\begin{align*}\tau(\mathbf{k},\mathbf{k}',\mathbf{k}'') &= (k_{\tau({1})},\dots,k_{\tau({n+n'+n''})}),\\
\tau(\mathbf{l},\mathbf{l}',\mathbf{l}'') &=(l_{\tau({1})},\dots,l_{\tau({n+n'+n''})}),
\end{align*}
and
\begin{align*}
S_{n,n',n''} = \left\{\tau \in S_{n+n'+n''}\mathrel{}\middle|\mathrel{} \begin{aligned}
&\tau^{-1}(1)<\dots<\tau^{-1}(n),\\ &\tau^{-1}(n+1)<\dots<\tau^{-1}(n+n'),\\
&\tau^{-1}(n+n'+1)<\dots<\tau^{-1}(n+n'+n'')
\end{aligned}\right\}.
\end{align*}
Therefore, we obtain the desired proposition.
\end{proof}
\begin{example}\label{ex:shuffle}
We obtain
\begin{align*}
&{\rm Ls}_{1,3}^{(0,1)}(\sigma)\cdot{\rm Ls}_{2}^{(1)}(\sigma)\\
&=-\int_{\substack{0<\theta_{1,1}<\theta_{1,2}<\sigma \\ 0<\theta_{2,1}<\sigma }}\theta_{1,2}A(\theta_{1,2}) \theta_{2,1}\,d\theta_{1,1} d\theta_{1,2} d\theta_{2,1}\\
&=-\left(\int_{0<\theta_{2,1}<\theta_{1,1}<\theta_{1,2}<\sigma}
+\int_{0<\theta_{1,1}<\theta_{2,1}<\theta_{1,2}<\sigma}
+\int_{0<\theta_{1,1}<\theta_{1,2}<\theta_{2,1}<\sigma}\right)\\
&\quad\theta_{1,2}A(\theta_{1,2}) \theta_{2,1}\,d\theta_{1,1} d\theta_{1,2} d\theta_{2,1}\\
&={\rm Ls}_{2,1,3}^{(1,0,1)}(\sigma)+{\rm Ls}_{1,2,3}^{(0,1,1)}(\sigma)+{\rm Ls}_{1,3,2}^{(0,1,1)}(\sigma).
\end{align*}
\end{example}
The second proposition is as follows.
\begin{proposition}\label{pr:j}
For index $\mathbf{k}=(k_{1},\dots,k_{n})$ with $n \ge 2$, we define 
\begin{align*}
\mathbf{k}_{k, j+}&= (k_{1},\dots,k_{j-1}, k_{j+1}+k,k_{j+2},\dots,k_{n})&(j \in \{1,\dots,n-1\}),\\
\mathbf{k}_{k, j-}&= (k_{1},\dots,k_{j-2}, k_{j-1}+k,k_{j+1},\dots,k_{n})&(j \in \{2,\dots,n\}).
\end{align*}
For $\mathbf{k}=(k_{1},\dots,k_{n})$ and $j \in \{1,\dots,n\}$, if $k_j -1 - l_j = 0$, then the following formula holds:
\begin{align}\label{eq:j}
{\rm Ls}_{\mathbf{k}}^{\mathbf{l}}(\sigma)&=\begin{cases}
-\frac{1}{k_{1}}\sigma^{k_{1}}&{\rm if}\ n=1, j=1,\\
-\frac{1}{k_1}{\rm Ls}_{\mathbf{k}_{k_{1},1+}}^{\mathbf{l}_{k_{1},1+}}(\sigma)&{\rm if}\ n \ge 2, j=1, \\
\frac{1}{k_j}\left({\rm Ls}_{\mathbf{k}_{k_{j},j-}}^{\mathbf{l}_{k_{j},j-}}(\sigma)
-{\rm Ls}_{\mathbf{k}_{k_{j},j+}}^{\mathbf{l}_{k_{j},j+}}(\sigma)\right)&{\rm if}\ n \ge 2, 1<j<n,\\
\frac{1}{k_n}\left({\rm Ls}_{\mathbf{k}_{k_{n},n-}}^{\mathbf{l}_{k_{n},n-}}(\sigma)-\sigma^{k_{n}}{\rm Ls}_{\mathbf{k}_{0,n-}}^{\mathbf{l}_{0,n-}}(\sigma)\right)&{\rm if}\ n \ge  2, j = n .
\end{cases}
\end{align}
\end{proposition}
\begin{remark}
When $n \ge 2$, if there exist two or more distinct $j$'s in the set $\{j \mid k_{j}-1-l_{j}=0, 1\le j \le n\}$, then we can use formula (\ref{eq:j}) with any of these $j$. Let $j_{1},\dots, j_{m}$ $(j_{1}<\dots<j_{m})$ be all elements of $\{j \mid k_{j}-1-l_{j}=0, 1\le j \le n\}$. If we use formula (\ref{eq:j}) with $j_{a}$ $(1\le a \le m)$, then on the right hand side, the indices of the log-sine integrals have depth $n-1$ and $j_{1},\dots,j_{a-1},j_{a+1}-1,\dots,j_{m}-1$ are all elements of $\{j \mid k_{j}-1-l_{j}=0, 1\le j \le n-1\}$. If we choose any $j_{1}, j_{2}$ $(j_{1}<j_{2})$ in $j_{1},\dots, j_{m}$,  then the expressions given by applying formula (\ref{eq:j}) first with $j_{1}$ and then with $j_{2}-1$, and those by applying formula (\ref{eq:j}) first with $j_{2}$ and then with $j_{1}$ are the same. Moreover, by applying formula (\ref{eq:j}) repeatedly, we can see that any ${\rm Ls}_{\mathbf{k}}^{\mathbf{l}}(\sigma)$ satisfying $k_{j} -1 - l_{j} \ge 0$ for all $ j \in \{1,\dots,n\}$ can be written as a $\mathbb{Q}$-linear combination of the products of $\sigma^m\ (m \ge 0)$ and an iterated log-sine integral at $\sigma$ satisfying $k_{j} -1 - l_{j} > 0$ for all $j \in \{1,\dots,n\}$.
\end{remark}
\begin{proof}
We have
\begin{align*}
{\rm Ls}_{\mathbf{k}}^{\mathbf{l}}(\sigma)&=(-1)^{n}\int_{0<\theta_{1}<\dots<\theta_{n}<\sigma}\theta_j^{l_j}\,d\theta_j
\prod_{\substack{1 \le u \le n \\ u \ne j}}\theta_u^{l_u}A^{k_u-1-l_u}(\theta_u)\,d\theta_u\\
&=(-1)^{n}\int_{D_{\sigma}^{(j)}}
\left(\frac{\theta_{j+1}^{l_j+1}-\theta_{j-1}^{l_j+1}}{l_j+1}\right)\prod_{\substack{1 \le u \le n \\ u \ne j}}\theta_u^{l_u}A^{k_u-1-l_u}(\theta_u)\,d\theta_u\\
&=\begin{cases}
-\frac{1}{k_{1}}\sigma^{k_{1}}&{\rm if}\ n=1, j=1,\\
-\frac{1}{k_1}{\rm Ls}_{\mathbf{k}_{k_{1},1+}}^{\mathbf{l}_{k_{1},1+}}(\sigma)&{\rm if}\ n \ge 2, j=1, \\
\frac{1}{k_j}\left({\rm Ls}_{\mathbf{k}_{k_{j},j-}}^{\mathbf{l}_{k_{j},j-}}(\sigma)
-{\rm Ls}_{\mathbf{k}_{k_{j},j+}}^{\mathbf{l}_{k_{j},j+}}(\sigma)\right)&{\rm if}\ n \ge 2, 1<j<n,\\
\frac{1}{k_n}\left({\rm Ls}_{\mathbf{k}_{k_{n},n-}}^{\mathbf{l}_{k_{n},n-}}(\sigma)-\sigma^{k_{n}}{\rm Ls}_{\mathbf{k}_{0,n-}}^{\mathbf{l}_{0,n-}}(\sigma)\right)&{\rm if}\ n \ge  2, j = n,
\end{cases}
\end{align*}
where the domain of integration $D_{\sigma}^{(j)}$ is defined by
\[D_{\sigma}^{(j)}=\{0=\theta_0<\dots<\theta_{j-1}<\theta_{j+1}<\dots<\theta_{n+1}=\sigma\}.\]
Therefore, we obtain the desired proposition.
\end{proof}
\begin{example}\label{ex:j}
We obtain
\begin{align*}
{\rm Ls}_{1,3}^{(0,1)}(\sigma) = -{\rm Ls}_{4}^{(2)}(\sigma), \quad { \rm Ls}_{2}^{(1)}(\sigma) = -\frac{\sigma^{2}}{2}.
\end{align*}
By applying formula (\ref{eq:j}) to ${\rm Ls}_{2,1,3}^{(1,0,1)}(\sigma)$ with $j=1$ and $j=2$, we obtain
\begin{align*}
{\rm Ls}_{2,1,3}^{(1,0,1)}(\sigma) = -\frac{1}{2}{\rm Ls}_{3,3}^{(2,1)}(\sigma),\quad {\rm Ls}_{2,1,3}^{(1,0,1)}(\sigma) = {\rm Ls}_{3,3}^{(2,1)}(\sigma) - {\rm Ls}_{2,4}^{(1,2)}(\sigma),
\end{align*}
respectively. Since we obtain
\begin{align*}
{\rm Ls}_{3,3}^{(2,1)}(\sigma) = -\frac{1}{3}{\rm Ls}_{6}^{(4)}(\sigma),\quad {\rm Ls}_{2,4}^{(1,2)}(\sigma) = -\frac{1}{2}{\rm Ls}_{6}^{(4)}(\sigma),
\end{align*}
we also obtain
\begin{align*}
{\rm Ls}_{2,1,3}^{(1,0,1)}(\sigma) = \frac{1}{6}{\rm Ls}_{6}^{(4)}(\sigma),
\end{align*}
independently of the choice of order. In the same way, we obtain
\begin{align*}
{\rm Ls}_{1,2,3}^{(0,1,1)}(\sigma) = \frac{1}{3}{\rm Ls}_{6}^{(4)}(\sigma), \quad {\rm Ls}_{1,3,2}^{(0,1,1)}(\sigma) = -\frac{1}{2}{\rm Ls}_{6}^{(4)}(\sigma)+\frac{\sigma^2}{2}{\rm Ls}_{4}^{(2)}(\sigma).
\end{align*}
\end{example}
By Proposition \ref{pr:Lsshuffle} and Proposition \ref{pr:j}, we obtain the following theorem.
\begin{theorem}\label{pr:prod}
The product of iterated log-sine integrals at $\sigma$ satisfying $k_{j} -1 - l_{j} \ge 0$ for all $ j \in \{1,\dots,n\}$ can be written as a $\mathbb{Q}$-linear combination of the products of $\sigma^m\ (m \ge 0)$ and an iterated log-sine integral at $\sigma$ satisfying $k_{j} -1 - l_{j} > 0$ for all $j \in \{1,\dots,n\}$. Moreover, this expression does not depend on the order of applying formula (\ref{eq:shuffle}) and formula (\ref{eq:j}).
\end{theorem}
\begin{proof}
We prove only that the computations by using formula (\ref{eq:shuffle}) and  by using formula (\ref{eq:j}) is commutative. 
It is enough to show that when $k_{j} - 1 - l_{j} = 0$,
the expressions given by applying first formula (\ref{eq:shuffle}) and then formula (\ref{eq:j}), and those by applying first formula (\ref{eq:j}) and then formula (\ref{eq:shuffle}) are the same.
We assume $1 < j < n$. Also in the case of $j=1$ and $j=n$, we can prove it in the same way.
By applying formula (\ref{eq:shuffle}) to ${\rm Ls}_{\mathbf{k}}^{\mathbf{l}}(\sigma)\cdot {\rm Ls}_{\mathbf{k}'}^{\mathbf{l}'}(\sigma)$, we have
\begin{align*}
{\rm Ls}_{\mathbf{k}}^{\mathbf{l}}(\sigma)\cdot {\rm Ls}_{\mathbf{k}'}^{\mathbf{l}'}(\sigma) = \sum_{\tau \in S_{n, n'}}{\rm Ls}_{\tau(\mathbf{k},\mathbf{k}')}^{\tau(\mathbf{l},\mathbf{l}')}(\sigma).
\end{align*}
Here, we define
\begin{align*}
H_{j} = \left\{\begin{aligned}
&(\mathbf{h}_{0,a},\mathbf{h}_{a,b},\mathbf{h}_{b,c},\mathbf{h}_{c,n+n'})
\end{aligned}\mathrel{}\middle|\mathrel{} \substack{1 \le a < b < c\le n+n'\\ \tau \in S_{n, n'}\\ \mathbf{h}_{0,a}=(\tau(1),\dots,\tau(a-1))\\ \tau(a)=j-1\\ \mathbf{h}_{a,b} = (\tau(a+1),\dots,\tau(b-1))\\\tau(b)=j \\ \mathbf{h}_{b,c} = (\tau(b+1),\dots,\tau(c-1)\\\tau(c)=j+1\\ \mathbf{h}_{c,n+n'} = (\tau(c+1),\dots,\tau(n+n')) \\}\right\},
\end{align*}
\begin{align*}
H'_{j} &= \{(\mathbf{h}_{0,a},\mathbf{h}_{c,n+n'})\mid (\mathbf{h}_{0,a},\mathbf{h}_{a,b},\mathbf{h}_{b,c},\mathbf{h}_{c,n+n'}) \in H_{j}\},
\end{align*}
and for fixed $\mathbf{h}_{0,a}$ and $\mathbf{h}_{c,n+n'}$, 
\begin{align*}
H_{j}^{\mathbf{h}_{0,a},\mathbf{h}_{c,n+n'}} &= \{(\mathbf{h}'_{a,b},\mathbf{h}'_{b,c})\mid (\mathbf{h}_{0,a},\mathbf{h}'_{a,b},\mathbf{h}'_{b,c},\mathbf{h}_{c,n+n'}) \in H_{j}\}.
\end{align*}
Moreover, for $\mathbf{h}_{i,j} = (h_{i+1}, \dots, h_{j-1})$, we define
\begin{align*}
l_{\mathbf{h}_{i,j}}=(l_{h_{i+1}},\dots,l_{h_{j-1}}),\quad k_{\mathbf{h}_{i,j}}=(k_{h_{i+1}},\dots,k_{h_{j-1}}).
\end{align*}
Then we have
\begin{align*}
&\sum_{\tau \in S_{n, n'}}{\rm Ls}_{\tau(\mathbf{k},\mathbf{k}')}^{\tau(\mathbf{l},\mathbf{l}')}(\sigma)\\
&=\sum_{(\mathbf{h}_{0,a},\mathbf{h}_{c,n+n'}) \in H'_{j}}\sum_{(\mathbf{h}_{a,b},\mathbf{h}_{b,c}) \in H_{j}^{\mathbf{h}_{0,a},\mathbf{h}_{c,n+n'}}} {\rm Ls}_{k_{\mathbf{h}_{0,a}}, k_{j-1}, k_{\mathbf{h}_{a,b}}, k_{j}, k_{\mathbf{h}_{b,c}}, k_{j+1}, k_{\mathbf{h}_{c,n+n'}}}^{l_{\mathbf{h}_{0,a}}, l_{j-1}, l_{\mathbf{h}_{a,b}}, l_{j}, l_{\mathbf{h}_{b,c}}, l_{j+1}, l_{\mathbf{h}_{c,n+n'}}}(\sigma).
\end{align*}
Here, if we regard $(\mathbf{h}_{a,b},\mathbf{h}_{b,c}) \in \mathbb{N}^{b-a-1}\times \mathbb{N}^{c-b-1}$ as $(\mathbf{h}_{a,b},\mathbf{h}_{b,c}) \in \mathbb{N}^{c-a-2}$, then all $(\mathbf{h}_{a,b},\mathbf{h}_{b,c}) \in H_{j}^{\mathbf{h}_{0,a},\mathbf{h}_{c,n+n'}}$ can be regarded as the same $(\mathbf{h}_{a,b},\mathbf{h}_{b,c}) \in \mathbb{N}^{c-a-2}$. Therefore we denote such $(\mathbf{h}_{a,b},\mathbf{h}_{b,c}) \in \mathbb{N}^{c-a-2}$ by $\mathbf{h}_{a,c} \in \mathbb{N}^{c-a-2}$.
By applying formula (\ref{eq:j}) , the above equation is equal to
\begin{align*}
&\sum_{(\mathbf{h}_{0,a},\mathbf{h}_{c,n+n'}) \in H'_{j}} \frac{1}{k_{j}}\left({\rm Ls}_{k_{\mathbf{h}_{0,a}}, k_{j-1}+k_{j}, k_{\mathbf{h}_{a,c}}, k_{j+1}, k_{\mathbf{h}_{c,n+n'}}}^{l_{\mathbf{h}_{0,a}}, l_{j-1}+k_{j}, l_{\mathbf{h}_{a,c}}, l_{j+1}, l_{\mathbf{h}_{c,n+n'}}}(\sigma)\right.\\
&\qquad\qquad\qquad\qquad\qquad\qquad\qquad\left.-{\rm Ls}_{k_{\mathbf{h}_{0,a}}, k_{j-1}, k_{\mathbf{h}_{a,c}}, k_{j}+k_{j+1}, k_{\mathbf{h}_{c,n+n'}}}^{l_{\mathbf{h}_{0,a}}, l_{j-1}, l_{\mathbf{h}_{a,c}}, k_{j}+l_{j+1}, l_{\mathbf{h}_{c,n+n'}}}(\sigma)\right)\\
&=\frac{1}{k_{j}}\left(\sum_{\tau \in S_{n-1,n'}}{\rm Ls}_{\tau(\mathbf{k}_{k_{j},j-},\mathbf{k}')}^{\tau(\mathbf{l}_{k_{j},j-},\mathbf{l}')}(\sigma)-\sum_{\tau \in S_{n-1,n'}}{\rm Ls}_{\tau(\mathbf{k}_{k_{j},j+},\mathbf{k}')}^{\tau(\mathbf{l}_{k_{j},j+},\mathbf{l}')}(\sigma)\right).
\end{align*}

On the other hand, by applying formula (\ref{eq:j}) to ${\rm Ls}_{\mathbf{k}}^{\mathbf{l}}(\sigma)\cdot {\rm Ls}_{\mathbf{k}'}^{\mathbf{l}'}(\sigma)$, we have
\begin{align*}
{\rm Ls}_{\mathbf{k}}^{\mathbf{l}}(\sigma)\cdot {\rm Ls}_{\mathbf{k}'}^{\mathbf{l}'}(\sigma)=
\frac{1}{k_j}\left({\rm Ls}_{\mathbf{k}_{k_{j},j-}}^{\mathbf{l}_{k_{j},j-}}(\sigma)-{\rm Ls}_{\mathbf{k}_{k_{j},j+}}^{\mathbf{l}_{k_{j},j+}}(\sigma)\right)\cdot {\rm Ls}_{\mathbf{k}'}^{\mathbf{l}'}(\sigma).
\end{align*}
By formula (\ref{eq:shuffle}), we have
\begin{align*}
{\rm Ls}_{\mathbf{k}}^{\mathbf{l}}(\sigma)\cdot {\rm Ls}_{\mathbf{k}'}^{\mathbf{l}'}(\sigma)=
\frac{1}{k_{j}}\left(\sum_{\tau \in S_{n-1,n'}}{\rm Ls}_{\tau(\mathbf{k}_{k_{j},j-},\mathbf{k}')}^{\tau(\mathbf{l}_{k_{j},j-},\mathbf{l}')}(\sigma)-\sum_{\tau \in S_{n-1,n'}}{\rm Ls}_{\tau(\mathbf{k}_{k_{j},j+},\mathbf{k}')}^{\tau(\mathbf{l}_{k_{j},j+},\mathbf{l}')}(\sigma)\right).
\end{align*}
Therefore this theorem is proved.
\end{proof}
\begin{example}
For example, by applying formula (\ref{eq:shuffle}) to ${\rm Ls}_{1,3}^{(0,1)}(\sigma)\cdot{\rm Ls}_{2}^{(1)}(\sigma)$, we have
\begin{align*}
{\rm Ls}_{1,3}^{(0,1)}(\sigma)\cdot{\rm Ls}_{2}^{(1)}(\sigma)={\rm Ls}_{2,1,3}^{(1,0,1)}(\sigma)+{\rm Ls}_{1,2,3}^{(0,1,1)}(\sigma)+{\rm Ls}_{1,3,2}^{(0,1,1)}(\sigma)
\end{align*}
 (see Example \ref{ex:shuffle}). By applying formula (\ref{eq:j}) to each term of the right hand side, we obtain
\begin{align*}
{\rm Ls}_{1,3}^{(0,1)}(\sigma)\cdot{\rm Ls}_{2}^{(1)}(\sigma) = \frac{\sigma^2}{2}{\rm Ls}_{4}^{(2)}(\sigma)
\end{align*}
(see Example \ref{ex:j}). On the other hand, by applying formula (\ref{eq:j}) to ${\rm Ls}_{1,3}^{(0,1)}(\sigma)$ and ${\rm Ls}_{2}^{(1)}(\sigma)$, we obtain the same result:
\begin{align*}
{\rm Ls}_{1,3}^{(0,1)}(\sigma)\cdot{\rm Ls}_{2}^{(1)}(\sigma) = -{\rm Ls}_{4}^{(2)}(\sigma)\cdot-\frac{\sigma^{2}}{2} = \frac{\sigma^2}{2}{\rm Ls}_{4}^{(2)}(\sigma)
\end{align*}
(see Example \ref{ex:j}).
\end{example}
\begin{theorem}\label{th:LitoLs}
For $\mathbf{k}=(k_{1},\dots,k_{n})\in\mathbb{N}^{n}$ and $0<\sigma<\pi$, we have
\begin{align*}
&\mathrm{Li}_{\mathbf{k}}(1-e^{\pm i\sigma})\\
&=(\mp i)^n\int_{0<\theta_{1}<\dots<\theta_{n}<\theta_{n+1}=\sigma}\prod_{u=1}^{n} \frac{\left(A(\theta_{u+1}) - A(\theta_{u}) \pm \frac{i\theta_{u+1}}{2} \mp \frac{i\theta_{u}}{2}\right)^{k_{u}-1}}{(k_{u}-1)!}\,d\theta_{u}.
\end{align*}
\end{theorem}
Since $e^{\pi i/3} = 1-e^{-\pi i/3}$, Corollary \ref{th:LitoLspi/3} holds.
We can see that any $\mathrm{Li}_{\mathbf{k}}(1-e^{\pm i\sigma})$ can be written as a $\mathbb{Q}(i)$-linear combination of iterated log-sine integrals at $\sigma$ satisfying $k_{j} -1 - l_{j} \ge 0$ for all $ j \in \{1,\dots,n\}$. 
\begin{proof}
For $z=t_{n+1} \in \mathbb{C}\setminus[1,\infty)$, the multiple polylogarithm is written as
\begin{align*}
&\mathrm{Li}_{\mathbf{k}}(t_{n+1})=\int_{\mathbf{C}}
\prod_{u=1}^{n}\frac{1}{(k_{u}-1)!}\left(\int_{t_{u}}^{t_{u+1}}\frac{dt}{t}\right)^{k_{u}-1}\frac{dt_{u}}{1-t_{u}}.
\end{align*}
Here, $\mathbf{C}$ means the integral with respect to $t_{1},\dots,t_{n}$ in order on a path $0$ to $t_{n+1}$ which does not pass $[1,\infty)$. Namely, for a path $\gamma : [0,1] \rightarrow \mathbb{C}\setminus[1,\infty)$, $\gamma(0)=0, \gamma(1)=t_{n+1}$,
\[\mathbf{C} = \left\{(t_{1},\dots,t_{n}) \in (\mathbb{C}\setminus[1,\infty))^{n}\mathrel{}\middle|\mathrel{} \begin{aligned}&0<a_{1}<\dots<a_{n}<1\\ &\gamma(a_{1}) = t_{1},\dots,\gamma(a_{n})=t_{n}\end{aligned}\right\}.\]
We put $\gamma(a) = 1-e^{\pm ai\sigma}$ and $t_{u} = \gamma(\theta_{u}/\sigma)= 1-e^{\pm i\theta_{u}}$. Then we have
\begin{align*}
&\mathrm{Li}_{\mathbf{k}}(1-e^{\pm i\sigma})\\
=&\int_{0<\theta_{1}<\dots<\theta_{n}<\theta_{n+1}=\sigma}\prod_{u=1}^{n}\frac{1}{(k_{u}-1)!}
\left(\int_{1-e^{\pm i\theta_{u}}}^{1-e^{\pm i\theta_{u+1}}}\frac{dt}{t}\right)^{k_{u}-1}(\mp i)\,d\theta_{u}\\
=&(\mp i)^{n}
\int_{0<\theta_{1}<\dots<\theta_{n}<\theta_{n+1}=\sigma}\prod_{u=1}^{n}\frac{\left(\log(1-e^{\pm i\theta_{u+1}})-\log(1-e^{\pm i\theta_{u}})\right)^{k_{u}-1}}{(k_{u}-1)!}\,d\theta_{u}.
\end{align*}
Finally, we use the formula
\begin{align*}
\log(1-e^{\pm i\theta})&= A(\theta)\pm i\frac{\theta-\pi}{2}\qquad(0<\theta<2\pi),
\end{align*}
and the proof is complete.
\end{proof}
By using Theorem \ref{th:LitoLs}, we can give another proof of the following theorem.
\begin{theorem}[{Borwein, Broadhurst and Kamnitzer \cite[theorem 4.5]{BBK}}]
For nonnegative integers a and b,
\[{\rm mgl}(\{1\}^{a},2,\{1\}^{b}) = (-1)^{a+b+1}\frac{(\pi/3)^{a+b+2}}{2(a+b+2)!},\]
where {\rm mgl} is defined by
\[{\rm mgl}(a_{1},\dots,a_{k}) = \Re (i^{a_{1}+\dots+a_{k}} \mathrm{Li}_{a_{1},\dots,a_{k}}(e^{\frac{\pi}{3}i})).\]
\end{theorem}
\begin{proof}[Another proof]
By using Theorem \ref{th:LitoLs}, we have
\begin{align*}
&\mathrm{Li}_{\{1\}^{a},2,\{1\}^{b}}(e^{\frac{\pi}{3}i})\\
&=i^{a+b+1}\int_{0<\theta_{1}<\dots<\theta_{a+b+1}<\frac{\pi}{3}} A(\theta_{a+2}) - A(\theta_{a+1}) - \frac{i\theta_{a+2}}{2} + \frac{i\theta_{a+1}}{2}\,d\theta_{1}\cdots d\theta_{a+b+1}.
\end{align*}
Therefore, we have
\begin{align*}
&\Re\left( i^{a+b+2}\mathrm{Li}_{\{1\}^{a},2,\{1\}^{b}}(e^{\frac{\pi}{3}i})\right)\\
&=\frac{(-1)^{a+b+1}}{2}\int_{0=\theta_{0}<\theta_{1}<\dots<\theta_{a+b+1}<\theta_{a+b+2}=\frac{\pi}{3}}  \theta_{a+2} - \theta_{a+1}\,d\theta_{1}\cdots d\theta_{a+b+1}\\
&=\frac{(-1)^{a+b+1}}{2}\left(\frac{a+2}{(a+b+2)!}\left(\frac{\pi}{3}\right)^{a+b+2}-\frac{a+1}{(a+b+2)!}\left(\frac{\pi}{3}\right)^{a+b+2}\right)\\
&=(-1)^{a+b+1}\frac{(\pi/3)^{a+b+2}}{2(a+b+2)!}
\end{align*}
\end{proof}
\section{Applications to the theory of MZVs}\label{se:zeta}
In this section, we provide examples, a result obtained by computer and alternative proofs of some facts.
First, note that combining Corollay \ref{th:LitoLspi/3}, Corollary \ref{pr:prodpi/3} and Theorem \ref{th:BBK} for the case $\mathbf{k} = (\{1\}^{a-1},b+1)$ gives the following formula, which was provided in \cite{BBK}. 
\begin{proposition} For $a, b \in \mathbb{N}$, we have
\begin{align}
&\zeta(\{1\}^{a-1},b+1) \label{eq:log-sine weight 1}\\
&=-\frac{1}{a!b!}\left(\frac{\pi}{3}i\right)^{a}\left(-\frac{\pi}{3}i\right)^{b}+\frac{i^{a}}{(a-1)!b!}\int_{0}^{\frac{\pi}{3}}\theta^{a-1}\left(-A(\theta)+i\frac{\theta-\pi}{2}\right)^{b}\,d\theta\nonumber \\
&\quad+\frac{(-i)^{b}}{a!(b-1)!}\int_{0}^{\frac{\pi}{3}}\theta^{b-1}\left(-A(\theta)-i\frac{\theta-\pi}{2}\right)^{a}\,d\theta . \nonumber
\end{align}
\end{proposition}
\subsection{Examples}
First, we calculate $\zeta(2)$. By using Theorem \ref{th:BBK} and Corollay \ref{th:LitoLspi/3}, we obtain
\begin{align*}
\zeta(2) &= \mathrm{Li}_{2}(e^{\frac{\pi}{3}i})+\mathrm{Li}_{1}(e^{\frac{\pi}{3}i})\overline{\mathrm{Li}_{1}(e^{\frac{\pi}{3}i})}+\overline{\mathrm{Li}_{2}(e^{\frac{\pi}{3}i})}\\
&=i\int_{0<\theta_{1}<\frac{\pi}{3}}\left(- A(\theta_{1}) - \frac{i\pi}{6} + \frac{i\theta_{1}}{2}\right)\,d\theta_{1}+\int_{0<\theta_{1}<\frac{\pi}{3}}\,d\theta_{1}\cdot\int_{0<\theta_{1}<\frac{\pi}{3}}\,d\theta_{1}\\
&\quad-i\int_{0<\theta_{1}<\frac{\pi}{3}}\left(- A(\theta_{1}) + \frac{i\pi}{6} - \frac{i\theta_{1}}{2}\right)\,d\theta_{1}\\
&= \frac{\pi^2}{6}.
\end{align*}
Similarly, by calculating $\zeta(3)$, we obtain
\begin{align}
&\zeta(3) =\frac{\pi}{2}{\rm Ls}_{2}^{(0)}\left(\frac{\pi}{3}\right)-\frac{3}{2}{\rm Ls}_{3}^{(1)}\left(\frac{\pi}{3}\right)-\frac{i}{2}{\rm Ls}_{3}^{(0)}\left(\frac{\pi}{3}\right)-\frac{7i}{216}\pi^{3} \label{eq:z3}.
\end{align}
By taking the real part of (\ref{eq:z3}), we have
\begin{align*}
\zeta(3) &=\frac{1}{2}\pi{\rm Ls}_{2}^{(0)}\left(\frac{\pi}{3}\right)-\frac{3}{2}{\rm Ls}_{3}^{(1)}\left(\frac{\pi}{3}\right),
\end{align*}
which can be seen in \cite[p.230]{L}.

Next, we calculate MZVs of weight $4$. The real parts of all MZVs of weight $4$ are as follows.
\begin{align}\label{eq:weight4}
\zeta(4)&=-\frac{23}{5184}\pi^4-\frac{1}{4}\pi{\rm Ls}_{3}^{(0)}\left(\frac{\pi}{3}\right)+\frac{1}{4}{\rm Ls}_{4}^{(1)}\left(\frac{\pi}{3}\right),\\
\zeta(1,1,2) &=-\frac{23}{5184}\pi^4-\frac{1}{4}\pi{\rm Ls}_{3}^{(0)}\left(\frac{\pi}{3}\right)+\frac{1}{4}{\rm Ls}_{4}^{(1)}\left(\frac{\pi}{3}\right),\nonumber\\
\zeta(1,3) &=\frac{7}{1296}\pi^4+{\rm Ls}_{4}^{(1)}\left(\frac{\pi}{3}\right),\nonumber\\
\zeta(2,2) &=\frac{1}{324}\pi^4-2{\rm Ls}_{4}^{(1)}\left(\frac{\pi}{3}\right).\nonumber
\end{align}
By taking the imaginary part of (\ref{eq:z3}), we obtain
\begin{align}
{\rm Ls}_{3}^{(0)}\left(\frac{\pi}{3}\right)=-\frac{7}{108}\pi^{3} \label{eq:ls30},
\end{align}
which can be seen in \cite{V} and \cite{L}.
Moreover, by taking the imaginary part of the following expression of $\zeta(1,4)$:
\begin{align*}
\zeta(1,4)=-\frac{17}{25920} i \, \pi^{5} + \frac{1}{8} \, \pi^{2} {\rm Ls}_{3}^{(1)}\left(\frac{\pi}{3}\right) - \frac{1}{2} \, \pi {\rm Ls}_{4}^{(2)}\left(\frac{\pi}{3}\right)\\
- \frac{1}{4} i \, \pi {\rm Ls}_{4}^{(1)}\left(\frac{\pi}{3}\right) + \frac{3}{8} \, {\rm Ls}_{5}^{(3)}\left(\frac{\pi}{3}\right) - \frac{1}{6} \, {\rm Ls}_{5}^{(1)}\left(\frac{\pi}{3}\right),
\end{align*}
we obtain
\begin{align}
{\rm Ls}_{4}^{(1)}\left(\frac{\pi}{3}\right)=-\frac{17}{6480}\pi^{4}, \label{eq:ls41}
\end{align}
which can be seen in \cite{V} and \cite{L}.
By applying (\ref{eq:ls30}) and (\ref{eq:ls41}) to (\ref{eq:weight4}), we obtain 
\[\zeta(4) = \frac{\pi^4}{90},\quad
\zeta(1,1,2) = \frac{\pi^4}{90},\quad
\zeta(1,3) =\frac{\pi^4}{360},\quad
\zeta(2,2) =\frac{\pi^4}{120}.\]
Therefore, we obtain a relation among multiple zeta values:
\[3\zeta(4)=3\zeta(1,1,2)=12\zeta(1,3)=4\zeta(2,2).\]
In this way, by applying relations among iterated log-sine integrals to the iterated log-sine integral expression of multiple zeta values, we obtain relations among multiple zeta values.
\subsection{A result obtained by using computer}
We compute the number of relations among MZVs obtained by our method. Let $\mathcal{Z}_{k}$ be the $\mathbb{Q}$-vector space generated by all MZVs of weight $k$, and $d_{0}=1, d_{1}=0, d_{2}=1$ and $d_{k} =  d_{k-2}+d_{k-3}$ $(k \ge 3)$.
Then Zagier's conjecture states $\dim \mathcal{Z}_{k} = d_{k}$ (see \cite{Z}).
Note that $\dim \mathcal{Z}_{k} \le d_{k}$ was proved in \cite{T} and \cite{DG}. Let $l_{k}$ be the upper bound of $\dim \mathcal{Z}_{k}$ obtained by our method, that is the number of all MZVs of weight $k$ minus the number of independent relations among MZVs obtained by our method. The result of the calculation of $l_{k}$ by using computer is as follows. Here, $2^{k-2}$ is the number of all MZVs of weight $k$.
\begin{table}[htb]
  \begin{tabular}{|c|c|c|c|c|c|c|c|c|c|c} \hline
     $k$ & 2 & 3 & 4 &5 &6 &7&8&9&10&$\cdots$\\ \hline 
    $d_{k}$ & 1 & 1 & 1 &2&2&3&4&5&7&$\cdots$\\ \hline 
    $l_{k}$ & 1 & 1 & 1 &2&2&4&4&9&9&$\cdots$\\ \hline 
    $2^{k-2}$ & 1 & 2 & 4 &8&16&32&64&128&256&$\cdots$\\ \hline
  \end{tabular}
\end{table}

We can see $d_{k}$ equals $l_{k}$ up to weight $6$. However, $d_{k}$ does not equal to $l_{k}$ when the weight is $7$. Therefore, we can not obtain all relations among MZVs by using our method.  However, by comparing $d_{k}$, $l_{k}$ and $2^{k-2}$, we can see that many relations are obtained by our method. 

 The author's program is implemented in the computer algebra system SageMath. The details of calculation algorithm are explained in Section \ref{se:algo}. Note that there is a possibility that more relations can be obtained, because the author's program is only to use relations among iterated log-sine integrals obtained by calculations of weight up to $k+1$ for obtaining relations among MZVs of weight $k$. However, the last one relation of MZVs of weight $7$ can not obtained by relations among iterated log-sine integrals obtained by calculations of weight $8$, $9$, $10$ and $11$ with the author's program.

\subsection{Alternative proofs of several known results}
In this subsection, we give alternative proofs of several known results on multiple zeta values by using our method.

The first is the duality of multiple zeta values, that is $\zeta(\mathbf{k}) = \zeta(\mathbf{k}^*)$. We can prove $\zeta(\mathbf{k})$ and $\zeta(\mathbf{k}^*)$ have the same log-sine integral expressions. It is stated that Theorem \ref{th:BBK} gives another proof of the duality in \cite{BBK}. Namely, on Theorem \ref{th:BBK},  the right hand side of $\zeta(\mathbf{k})$ is equal to the complex conjugate of the right hand side of $\zeta(\mathbf{k}^*)$. Therefore, by using Corollary \ref{th:LitoLspi/3} and Corollary \ref{pr:prodpi/3} for each, we can see that $\zeta(\mathbf{k})$ and $\zeta(\mathbf{k}^*)$ have the same log-sine integral expressions.

The second is Euler's result on $\zeta(2k)$. The following result is proven by using only our method. 
\begin{theorem}For $k \in \mathbb{N}$ and $|X| <1$, we have
\begin{align}
&\sum_{k=1}^{\infty}\Big(\pi\zeta(2k)-2i\sum_{u=1}^{k-1}\frac{1}{(2u)!}(-1)^{u}(1-2^{2u})B_{2u}\pi^{2u}\zeta(\{1\}^{2(k-u)-1},2)\label{eq:zetagf}\\
&-i\sum_{u=0}^{k-2}\frac{1}{(2u)!}(-1)^{u}(2-2^{2u})B_{2u}\pi^{2u}\sum_{v=1}^{k-u-1}\zeta(\{1\}^{2(k-v-u)-1},2v+2)\Big)X^{2k}\nonumber\\
&=\frac{\pi}{2}-\frac{\pi^{2}}{2}X\cot\left(\pi X\right)\nonumber\\
&\quad+i\pi X^{2}\csc(\pi X)\int_{0}^{\frac{\pi}{3}}\sin((\theta+\pi) X)A(\theta)\,d\theta\nonumber\\
&\quad-2i\pi X\csc(\pi X)\int_{0}^{\frac{\pi}{3}}\sinh(A(\theta)X)\sin\left(\frac{\theta+\pi}{2}X\right)\cos\left((\theta-\pi) X\right)\,d\theta,\nonumber
\end{align}
where $B_{k}$ denote the Bernoulli numbers.
In particular, we have
\[\zeta(2k)=(-1)^{k+1}\frac{(2\pi)^{2k}B_{2k}}{2(2k)!}.\]
\end{theorem}
\begin{proof}
Because
\begin{align}
\zeta(2k)=&\frac{(-1)^{k}}{(2k-1)!}\left(\frac{\pi}{3}\right)^{2k}+i\frac{(-1)^{k-1}}{(2k-2)!}\int_{0}^{\frac{\pi}{3}}\left(A(\theta)+i\frac{\theta-\pi}{2}\right)\theta^{2k-2}\,d\theta \nonumber \\
&-i\frac{1}{(2k-1)!}\int_{0}^{\frac{\pi}{3}}\left(A(\theta)-i\frac{\theta-\pi}{2}\right)^{2k-1}\,d\theta \nonumber
\end{align}
by (\ref{eq:log-sine weight 1}), we have
\begin{align*}
&\sum_{k=1}^{\infty}\pi\zeta(2k)X^{2k}\\
=&-\frac{\pi^2 X}{3}\sin\left(\frac{\pi X}{3}\right)+i\pi X^{2}\int_{0}^{\frac{\pi}{3}}\left(A(\theta)+i\frac{\theta-\pi}{2}\right)\cos\left(\theta X\right)\,d\theta \nonumber \\
&-i\pi X \int_{0}^{\frac{\pi}{3}}\sinh\left(\left(A(\theta)-i\frac{\theta-\pi}{2}\right)X\right)\,d\theta. \nonumber
\end{align*}
Because
\begin{align*}
&\zeta(\{1\}^{2(k-u)-1},2)\\
=&i\frac{(-1)^{k-u}}{(2(k-u))!}\left(\frac{\pi}{3}\right)^{2(k-u)+1}-\frac{i}{(2(k-u))!}\int_{0}^{\frac{\pi}{3}}\left(A(\theta)+i\frac{\theta-\pi}{2}\right)^{2(k-u)}\,d\theta \nonumber \\
&-\frac{i}{(2(k-u)-1)!}\int_{0}^{\frac{\pi}{3}}(i\theta)^{2(k-u)-1}\left(A(\theta)-i\frac{\theta-\pi}{2}\right)\,d\theta \nonumber
\end{align*}
by (\ref{eq:log-sine weight 1}), we have
\begin{align*}
&-2i\sum_{k=1}^{\infty}\sum_{u=1}^{k-1}\frac{1}{(2u)!}(-1)^{u}(1-2^{2u})B_{2u}\pi^{2u}\zeta(\{1\}^{2(k-u)-1},2)X^{2k}\\
&=\frac{\pi^{2}X}{3}\tan\left(\frac{\pi X}{2}\right)\cos\left(\frac{\pi X}{3}\right)-\frac{\pi^{2}X}{3}\tan\left(\frac{\pi X}{2}\right)\\
&-\pi X\tan\left(\frac{\pi}{2}X\right)\int_{0}^{\frac{\pi}{3}}\cosh\left(\left(A(\theta)+i\frac{\theta-\pi}{2}\right)X\right)\,d\theta+\frac{\pi^{2}X}{3}\tan\left(\frac{\pi X}{2}\right)\\
&-i\pi X^{2} \tan\left(\frac{\pi}{2}X\right)\int_{0}^{\frac{\pi}{3}}\sin\left(\theta X\right)\left(A(\theta)-i\frac{\theta-\pi}{2}\right)\,d\theta.
\end{align*}
Because
\begin{align*}
&\zeta(\{1\}^{2(k-v-u)-1},2v+2) \\
&=i\frac{(-1)^{k-u}}{(2(k-v-u))!(2v)!}\left(\frac{\pi}{3}\right)^{2(k-v-u)}\int_{0}^{\frac{\pi}{3}}\theta^{2v}\,d\theta\\
&-\frac{i}{(2(k-v-u))!(2v)!}\int_{0}^{\frac{\pi}{3}}\left(-A(\theta)-i\frac{\theta-\pi}{2}\right)^{2(k-v-u)}(-i\theta)^{2v}\,d\theta \nonumber \\
&+\frac{i}{(2(k-v-u)-1)!(2v+1)!}\int_{0}^{\frac{\pi}{3}}(i\theta)^{2(k-v-u)-1}\left(-A(\theta)+i\frac{\theta-\pi}{2}\right)^{2v+1}\,d\theta \nonumber
\end{align*}
by (\ref{eq:log-sine weight 1}), we have
\begin{align*}
&-i\sum_{k=1}^{\infty}\sum_{u=0}^{k-2}\frac{1}{(2u)!}(-1)^{u}(2-2^{2u})B_{2u}\pi^{2u}\sum_{v=1}^{k-u-1}\zeta(\{1\}^{2(k-v-u)-1},2v+2)X^{2k}\\
&=\frac{\pi X}{2}\csc\left(\pi X\right)\int_{0}^{\frac{\pi}{3}}\cos\left(\left(\theta-\frac{\pi}{3}\right)X\right)\,d\theta+\frac{\pi X}{2}\csc\left(\pi X\right)\int_{0}^{\frac{\pi}{3}}\cos\left(\left(\theta+\frac{\pi}{3}\right)X\right)\,d\theta\\
&-\frac{\pi^{2} X}{3}\csc\left(\pi X\right)\cos\left(\frac{\pi}{3}X\right)-\pi X\csc\left(\pi X\right)\int_{0}^{\frac{\pi}{3}}\cos\left(\theta X\right)\,d\theta
+\frac{\pi^{2}}{3}X\csc\left(\pi X\right)\\
&-\frac{\pi X}{2} \csc\left(\pi X\right)\int_{0}^{\frac{\pi}{3}}\cosh\left(\left(A(\theta)+i\frac{\theta-\pi}{2}+i\theta\right)X\right)\,d\theta\\
&-\frac{\pi X}{2} \csc\left(\pi X\right)\int_{0}^{\frac{\pi}{3}}\cosh\left(\left(A(\theta)+i\frac{\theta-\pi}{2}-i\theta\right)X\right)\,d\theta\\
&+\pi X\csc\left(\pi X\right)\int_{0}^{\frac{\pi}{3}}\cosh\left(\left(A(\theta)+i\frac{\theta-\pi}{2}\right)X\right)\,d\theta+\pi X\csc\left(\pi X\right)\int_{0}^{\frac{\pi}{3}}\cos\left(\theta X\right)\,d\theta\\
&-\frac{\pi^{2}X}{3}\csc\left(\pi X\right)\\
&+\frac{\pi X}{2}\csc\left(\pi X\right)\int_{0}^{\frac{\pi}{3}}\cosh\left(\left(A(\theta)-i\frac{\theta-\pi}{2}-i\theta\right)X\right)\,d\theta\\
&-\frac{\pi X}{2}\csc\left(\pi X\right)\int_{0}^{\frac{\pi}{3}}\cosh\left(\left(A(\theta)-i\frac{\theta-\pi}{2}+i\theta\right)X\right)\,d\theta\\
&+i\pi X^{2}\csc\left(\pi X\right)\int_{0}^{\frac{\pi}{3}}\sin\left(\theta X\right)\left(A(\theta)-i\frac{\theta-\pi}{2}\right)\,d\theta.
\end{align*}
Here, we note that the following formulas hold:
\begin{align*}
&\cosh\left(\left(A(\theta)+i\frac{\theta-\pi}{2}+i\theta\right)X\right)+\cosh\left(\left(A(\theta)+i\frac{\theta-\pi}{2}-i\theta\right)X\right)\\
&-\cosh\left(\left(A(\theta)-i\frac{\theta-\pi}{2}-i\theta\right)X\right)+\cosh\left(\left(A(\theta)-i\frac{\theta-\pi}{2}+i\theta\right)X\right)\\
&=2i\sinh\left(A(\theta)X\right)\sin\left(\left(\frac{\theta-\pi}{2}+\theta\right)X\right)+2\cosh\left(A(\theta)X\right)\cos\left(\left(\frac{\theta-\pi}{2}-\theta\right)X\right),
\end{align*}
\begin{align*}
&\sinh\left(\left(A(\theta)-i\frac{\theta-\pi}{2}\right)X\right)\\
&=-i\sin\left(\frac{\theta-\pi}{2}X\right)\cosh\left(A(\theta)X\right)+\cos\left(\frac{\theta-\pi}{2}X\right)\sinh\left(A(\theta)X\right),
\end{align*}
\begin{align*}
&\cosh\left(\left(A(\theta)+i\frac{\theta-\pi}{2}\right)X\right)\\
&=\cos\left(\frac{\theta-\pi}{2}X\right)\cosh\left(A(\theta)X\right)+i\sin\left(\frac{\theta-\pi}{2}X\right)\sinh\left(A(\theta)X\right)
\end{align*}
and
\begin{align*}
&\tan\left(\frac{\pi X}{2}\right)-\csc\left(\pi X\right)=-\cot\left(\pi X\right).
\end{align*}
Therefore, the left hand side of (\ref{eq:zetagf}) is equal to
\begin{align*}
&-\frac{\pi^2 X}{3}\sin\left(\frac{\pi X}{3}\right)+i\pi X^{2}\int_{0}^{\frac{\pi}{3}}\left(A(\theta)+i\frac{\theta-\pi}{2}\right)\cos\left(\theta X\right)\,d\theta \nonumber \\
&-\pi X \int_{0}^{\frac{\pi}{3}}\sin\left(\frac{\theta-\pi}{2}X\right)\cosh\left(A(\theta)X\right)\,d\theta \nonumber\\
&-i\pi X \int_{0}^{\frac{\pi}{3}}\cos\left(\frac{\theta-\pi}{2}X\right)\sinh\left(A(\theta)X\right)\,d\theta \nonumber\\
&-\frac{\pi^{2}X}{3}\cot\left(\pi X\right)\cos\left(\frac{\pi X}{3}\right)\\
&+\pi X\cot\left(\pi X\right)\int_{0}^{\frac{\pi}{3}}\cos\left(\frac{\theta-\pi}{2}X\right)\cosh\left(A(\theta)X\right)\,d\theta\\
&+i\pi X\cot\left(\pi X\right)\int_{0}^{\frac{\pi}{3}}\sin\left(\frac{\theta-\pi}{2}X\right)\sinh\left(A(\theta)X\right)\,d\theta\\
&+i\pi X^{2} \cot\left(\pi X\right)\int_{0}^{\frac{\pi}{3}}\sin\left(\theta X\right)\left(A(\theta)-i\frac{\theta-\pi}{2}\right)\,d\theta\\
&+\frac{\pi X}{2}\csc\left(\pi X\right)\int_{0}^{\frac{\pi}{3}}\cos\left(\left(\theta-\frac{\pi}{3}\right)X\right)\,d\theta+\frac{\pi X}{2}\csc\left(\pi X\right)\int_{0}^{\frac{\pi}{3}}\cos\left(\left(\theta+\frac{\pi}{3}\right)X\right)\,d\theta\\
&-i\pi X \csc\left(\pi X\right)\int_{0}^{\frac{\pi}{3}}\sinh\left(A(\theta)X\right)\sin\left(\left(\frac{\theta-\pi}{2}+\theta\right)X\right)\,d\theta\\
&-\pi X \csc\left(\pi X\right)\int_{0}^{\frac{\pi}{3}}\cosh\left(A(\theta)X\right)\cos\left(\left(\frac{\theta-\pi}{2}-\theta\right)X\right)\,d\theta.
\end{align*}
Also notice that the following formulas hold:
\begin{align*}
&-\pi X \int_{0}^{\frac{\pi}{3}}\sin\left(\frac{\theta-\pi}{2}X\right)\cosh\left(A(\theta)X\right)\,d\theta\\
&-\pi X\csc\left(\pi X\right)\int_{0}^{\frac{\pi}{3}}\cosh\left(A(\theta)X\right)\cos\left(\left(\frac{\theta-\pi}{2}-\theta\right)X\right)\,d\theta\\
&=-\pi X\cot\left(\pi X\right)\int_{0}^{\frac{\pi}{3}}\cosh\left(A(\theta)X\right)\cos\left(\frac{\theta-\pi}{2}X\right)\,d\theta
\end{align*}
and
\begin{align*}
&-\pi X^{2} \int_{0}^{\frac{\pi}{3}}\frac{\theta-\pi}{2}\cos\left(\theta X\right)\,d\theta=\frac{\pi^{2} X}{3}\sin\left(\frac{\pi}{3}X\right)-\frac{\pi}{2}\cos\left(\frac{\pi}{3}X\right)+\frac{\pi}{2},\\
&\pi X^{2} \int_{0}^{\frac{\pi}{3}}\sin\left(\theta X\right)\frac{\theta-\pi}{2}\,d\theta=\frac{\pi^{2} X}{3}\cos\left(\frac{\pi}{3}X\right)-\frac{\pi^{2}}{2}X+\frac{\pi}{2}\sin\left(\frac{\pi}{3}X\right),\\
&\frac{\pi X}{2}\int_{0}^{\frac{\pi}{3}}\cos\left(\left(\theta-\frac{\pi}{3}\right)X\right)+\cos\left(\left(\theta+\frac{\pi}{3}\right)X\right)\,d\theta=\pi\cos\left(\frac{\pi}{3}X\right)\sin\left(\frac{\pi}{3}X\right).\\
\end{align*}
Therefore, the real part of the left hand side of (\ref{eq:zetagf}) is equal to
\begin{align*}
&-\frac{\pi}{2}\cos\left(\frac{\pi}{3}X\right)+\frac{\pi}{2}-\frac{\pi^{2}}{2}X\cot\left(\pi X\right)+\frac{\pi}{2}\sin\left(\frac{\pi}{3}X\right)\cot\left(\pi X\right)\\
&+\pi\csc\left(\pi X\right)\cos\left(\frac{\pi}{3}X\right)\sin\left(\frac{\pi}{3}X\right)\\
&=\frac{\pi}{2}-\frac{\pi^{2}}{2}X\cot\left(\pi X\right).
\end{align*}
Moreover, the following formulas hold:
\begin{align*}
&-i\pi X \int_{0}^{\frac{\pi}{3}}\cos\left(\frac{\theta-\pi}{2}X\right)\sinh\left(A(\theta)X\right)\,d\theta \nonumber\\
&+i\pi X\cot\left(\pi X\right)\int_{0}^{\frac{\pi}{3}}\sin\left(\frac{\theta-\pi}{2}X\right)\sinh\left(A(\theta)X\right)\,d\theta\\
&=i\pi X \csc\left(\pi X\right)\int_{0}^{\frac{\pi}{3}}\sinh\left(A(\theta)X\right)\sin\left(\left(\frac{\theta-\pi}{2}-\pi\right)X\right)\,d\theta
\end{align*}
and
\begin{align*}
&i\pi X^{2}\int_{0}^{\frac{\pi}{3}}A(\theta)\cos\left(\theta X\right)\,d\theta+i\pi X^{2} \cot\left(\pi X\right)\int_{0}^{\frac{\pi}{3}}\sin\left(\theta X\right)A(\theta)\,d\theta\\
&=i\pi X^{2}\csc\left(\pi X\right)\int_{0}^{\frac{\pi}{3}}A(\theta)\cos\left(\left(\theta+\pi\right) X\right)\,d\theta.
\end{align*}
Therefore, the imaginary part of the left hand side of (\ref{eq:zetagf}) is equal to
\begin{align*}
&i\pi X^{2}\csc\left(\pi X\right)\int_{0}^{\frac{\pi}{3}}A(\theta)\cos\left(\left(\theta+\pi\right) X\right)\,d\theta\\
&i\pi X \csc\left(\pi X\right)\int_{0}^{\frac{\pi}{3}}\sinh\left(A(\theta)X\right)\sin\left(\left(\frac{\theta-\pi}{2}-\pi\right)X\right)\,d\theta\\
&-i\pi X \csc\left(\pi X\right)\int_{0}^{\frac{\pi}{3}}\sinh\left(A(\theta)X\right)\sin\left(\left(\frac{\theta-\pi}{2}+\theta\right)X\right)\,d\theta.\\
\end{align*}
\end{proof}
As mentioned in the previous subsection, there are relations among MZVs not obtained by our method.
On our method, if all relations among iterated log-sine integrals are proved, then we can obtain all relations among MZVs in principle. Therefore, we are interested in the relations among iterated log-sine integrals which are not obtained by our method.

There is the following formula proved in \cite{L}:
\begin{align}\Re(\mathrm{Li}_{2k+1}(e^{\frac{\pi}{3}i})) &= \frac{1}{2}(1-2^{-2k})(1-3^{-2k})\zeta(2k+1).\label{eq:cr}
\end{align}
By applying Theorem \ref{th:BBK}, Corollary \ref{th:LitoLspi/3} and Corollary \ref{pr:prodpi/3} to the both sides, we obtain relations among iterated log-sine integrals, which could not be obtained by our method with the author's program when $k = 2, 3, 4$. However, even if we use this relation in addition to relations among iterated log-sine integrals obtained by the calculations of weight up to $2k+2$, the value of $l_{2k+1}$ does not change up to $k=4$.

However, by adding relations among log-sine integrals obtained by (\ref{eq:cr}) to our method, we can give another proof of the following formula.
\begin{theorem}[{Choi, Cho, and Srivistava \cite[(4.14)]{CCS}(Lewin  \cite[(7.160)]{L})}]
For $m \in \mathbb{Z}_{\ge 0}$, we have
\begin{align*}
&(-1)^{m}\int_{0}^{\frac{\pi}{3}}\left(\theta-\frac{\pi}{3}\right)^{2m+1}A(\theta)\,d\theta\\
&=-\frac{1}{2}(2m+1)!(1-2^{-2m-2})(1-3^{-2m-2})\zeta(2m+3)\\
&\quad+(2m+1)!\sum_{k=0}^{m}(-1)^{k}\left(\frac{\pi}{3}\right)^{2k}\frac{\zeta(2m+3-2k)}{(2k)!}.
\end{align*}
\end{theorem}
\begin{proof}
By (\ref{eq:log-sine weight 1}), we have
\begin{align*}
&\zeta(2m+3-k)\\
&=-\frac{1}{(2m+2-k)!}\left(\frac{\pi}{3}i\right)\left(-\frac{\pi}{3}i\right)^{2m+2-k}\\
&\quad-\frac{i}{(2m+1-k)!}\int_{0}^{\frac{\pi}{3}}\left(-A(\theta)-i\frac{\theta-\pi}{2}\right)(-i\theta)^{2m+1-k}\,d\theta \nonumber \\
&\quad+\frac{i}{(2m+2-k)!}\int_{0}^{\frac{\pi}{3}}\left(-A(\theta)+i\frac{\theta-\pi}{2}\right)^{2m+2-k}\,d\theta. \nonumber
\end{align*}
Therefore we have
\begin{align*}
&\sum_{k=0}^{2m+1}\frac{1}{k!}\left(\frac{\pi}{3}i\right)^{k}\zeta(2m+3-k)\\
&=-\frac{i}{(2m+1)!}\int_{0}^{\frac{\pi}{3}}\left(-A(\theta)-i\frac{\theta-\pi}{2}\right)\left(\frac{\pi}{3}i-i\theta\right)^{2m+1}\,d\theta \nonumber \\
&\quad+\frac{i}{(2m+2)!}\int_{0}^{\frac{\pi}{3}}\left(-A(\theta)+i\frac{\theta-\pi}{2}+\frac{\pi}{3}i\right)^{2m+2}\,d\theta. \nonumber
\end{align*}
By applying Corollary \ref{th:LitoLspi/3} to the second term of the right hand side and taking the real part of it, we obtain
\begin{align*}
&\sum_{k=0}^{m}\frac{(-1)^{k}}{(2k)!}\left(\frac{\pi}{3}\right)^{2k}\zeta(2m+3-2k)\\
&=\frac{(-1)^{m}}{(2m+1)!}\int_{0}^{\frac{\pi}{3}}A(\theta)\left(\theta-\frac{\pi}{3}\right)^{2m+1}\,d\theta+\Re\left(\mathrm{Li}_{2m+3}(e^{\frac{\pi}{3}i})\right).
\end{align*}
Finally, By applying (\ref{eq:cr}), we obtain the desired formula.
\end{proof}
\section{the algorithm of calculating $l_{k}$}\label{se:algo}
In this section, we explain the algorithm of calculating $l_{k}$. We use the example of the case $k=5$ for explanation. All calculations in this section are performed by using computer. First, we make a vertical vector consisting of all MZVs of weight $k$ (but when $\zeta(\mathbf{k})$ is included as an element of the vector, then we omit  $\zeta(\mathbf{k}^*)$). Next, we apply Theorem \ref{th:BBK}, Corollary \ref{th:LitoLspi/3} and Corollary \ref{pr:prodpi/3} to each entries, and write each entries as a $\mathbb{Q}$-linear combination of the products of $\pi^m\ (m \ge 0)$ and an iterated log-sine integral at $\pi/3$ satisfying $k_{j} -1 - l_{j} > 0$ for all $j \in \{1,\dots,n\}$, and taking those real parts. For the case $k=5$, we have \renewcommand{\arraystretch}{1.2}
\begin{align}
\left(\begin{array}{r}
\zeta\left(5\right) \\
\zeta\left(1, 4\right) \\
\zeta\left(2, 3\right) \\
\zeta\left(3, 2\right)
\end{array}\right)=\left(\begin{array}{r@{\kern0.5em}r@{\kern0.5em}r@{\kern0.5em}r@{\kern0.5em}r@{\kern0.5em}
r@{\kern0.5em}r@{\kern0.5em}r}
0 & 0 & -\frac{1}{12} & \frac{3}{16} & \frac{1}{12} & -\frac{1}{16} & \frac{1}{16} & -\frac{1}{48} \\
0 & 0 & -\frac{1}{6} & \frac{3}{8} & 0 & -\frac{1}{2} & \frac{1}{8} & 0 \\
0 & \frac{1}{2} & \frac{1}{2} & -\frac{7}{8} & 0 & \frac{5}{4} & -\frac{5}{12} & \frac{7}{216} \\
0 & -\frac{3}{2} & -\frac{1}{2} & \frac{3}{8} & 0 & -\frac{3}{4} & \frac{1}{4} & -\frac{1}{72}
\end{array}\right)\left(\begin{array}{r}
{\rm Ls}_{3 , 2}^{(0 , 0)}\left(\frac{\pi}{3}\right) \\
{\rm Ls}_{2 , 3}^{(0 , 0)}\left(\frac{\pi}{3}\right) \\
{\rm Ls}_{5}^{(1)}\left(\frac{\pi}{3}\right) \\
{\rm Ls}_{5}^{(3)}\left(\frac{\pi}{3}\right) \\
\pi {\rm Ls}_{4}^{(0)}\left(\frac{\pi}{3}\right) \\
\pi {\rm Ls}_{4}^{(2)}\left(\frac{\pi}{3}\right) \\
\pi^{2} {\rm Ls}_{3}^{(1)}\left(\frac{\pi}{3}\right) \\
\pi^{3} {\rm Ls}_{2}^{(0)}\left(\frac{\pi}{3}\right)
\end{array}\right)\label{eq:matrix5}.
\end{align}

Here, the vertical vector of the right hand side is composed of all $\pi^{m}{\rm Ls}_{\mathbf{k}}^{\mathbf{l}}\left(\frac{\pi}{3}\right)$ satisfying $m+|\mathbf{k}|=k$ and $\sum_{u=1}^{n}k_{u}-1-l_{u}\equiv1\ {\rm mod}\ 2$.

On the other hand, we make a vertical vector consisting of all MZVs of weight $k+1$ except when it is self-dual (but when $\zeta(\mathbf{k})$ is included as an element of the vector, then we omit  $\zeta(\mathbf{k}^*)$), and apply Theorem \ref{th:BBK}, Corollary \ref{th:LitoLspi/3} and Corollary \ref{pr:prodpi/3} to each entries, and write each entries as a $\mathbb{Q}$-linear combination of the products of $\pi^m\ (m \ge 0)$ and an iterated log-sine integral at $\pi/3$ satisfying $k_{j} -1 - l_{j} > 0$ for all $j \in \{1,\dots,n\}$, and taking those imaginary parts. Namely,
\begin{align}
\left(\begin{array}{r}
0 \\
0 \\
0 \\
0 \\
0 \\
0
\end{array}\right)=\left(\begin{array}{r}
\Im(\zeta\left(6\right)) \\
\Im(\zeta\left(1, 5\right)) \\
\Im(\zeta\left(2, 4\right)) \\
\Im(\zeta\left(3, 3\right)) \\
\Im(\zeta\left(4, 2\right)) \\
\Im(\zeta\left(1, 3, 2\right))
\end{array}\right) =A\left(\begin{array}{r}
{\rm Ls}_{2 , 2 , 2}^{(0 , 0 , 0)}\left(\frac{\pi}{3}\right) \\
{\rm Ls}_{4 , 2}^{(1 , 0)}\left(\frac{\pi}{3}\right) \\
{\rm Ls}_{3 , 3}^{(0 , 1)}\left(\frac{\pi}{3}\right) \\
{\rm Ls}_{3 , 3}^{(1 , 0)}\left(\frac{\pi}{3}\right) \\
{\rm Ls}_{2 , 4}^{(0 , 1)}\left(\frac{\pi}{3}\right) \\
{\rm Ls}_{6}^{(0)}\left(\frac{\pi}{3}\right) \\
{\rm Ls}_{6}^{(2)}\left(\frac{\pi}{3}\right) \\
{\rm Ls}_{6}^{(4)}\left(\frac{\pi}{3}\right) \\
\pi {\rm Ls}_{3 , 2}^{(0 , 0)}\left(\frac{\pi}{3}\right) \\
\pi {\rm Ls}_{2 , 3}^{(0 , 0)}\left(\frac{\pi}{3}\right) \\
\pi {\rm Ls}_{5}^{(1)}\left(\frac{\pi}{3}\right) \\
\pi {\rm Ls}_{5}^{(3)}\left(\frac{\pi}{3}\right) \\
\pi^{2} {\rm Ls}_{4}^{(0)}\left(\frac{\pi}{3}\right) \\
\pi^{2} {\rm Ls}_{4}^{(2)}\left(\frac{\pi}{3}\right) \\
\pi^{3} {\rm Ls}_{3}^{(1)}\left(\frac{\pi}{3}\right) \\
\pi^{4} {\rm Ls}_{2}^{(0)}\left(\frac{\pi}{3}\right)
\end{array}\right),\label{eq:matrix6}
\end{align}
where $A$ represents a matrix 
\begin{align*}
\left(\begin{array}{@{\kern0.3em}r@{\kern0.5em}r@{\kern0.5em}r@{\kern0.5em}r@{\kern0.5em}r@{\kern0.5em}r
@{\kern0.5em}r@{\kern0.5em}r@{\kern0.5em}r@{\kern0.5em}r@{\kern0.5em}
r@{\kern0.5em}r@{\kern0.5em}r@{\kern0.5em}r@{\kern0.5em}r@{\kern0.5em}r}
0 & 0 & 0 & 0 & 0 & \frac{1}{120} & -\frac{1}{48} & -\frac{5}{128} & 0 & 0 & \frac{1}{24} & -\frac{1}{96} & -\frac{1}{48} & \frac{1}{64} & -\frac{1}{96} & \frac{1}{384} \\
0 & 0 & 0 & 0 & 0 & 0 & -\frac{1}{12} & -\frac{1}{16} & 0 & 0 & \frac{1}{12} & \frac{1}{48} & 0 & \frac{1}{16} & -\frac{1}{48} & 0 \\
0 & 0 & 0 & 0 & \frac{1}{4} & 0 & \frac{7}{24} & \frac{11}{64} & 0 & -\frac{1}{4} & -\frac{1}{4} & -\frac{3}{16} & 0 & -\frac{1}{24} & \frac{1}{24} & -\frac{23}{5184} \\
0 & 0 & \frac{3}{4} & \frac{3}{4} & -\frac{3}{4} & 0 & -\frac{3}{8} & -\frac{9}{64} & -\frac{1}{4} & \frac{1}{2} & \frac{1}{4} & \frac{7}{16} & 0 & -\frac{1}{4} & \frac{7}{144} & -\frac{5}{1728} \\
0 & 0 & -\frac{3}{2} & -\frac{3}{2} & \frac{1}{2} & 0 & \frac{1}{4} & \frac{3}{32} & \frac{1}{2} & 0 & -\frac{1}{6} & -\frac{7}{24} & \frac{1}{36} & \frac{11}{48} & -\frac{1}{18} & \frac{7}{2592} \\
0 & 0 & 0 & 0 & -1 & 0 & -\frac{1}{4} & 0 & 0 & 0 & 0 & \frac{5}{24} & 0 & -\frac{13}{48} & \frac{1}{12} & -\frac{7}{1296}
\end{array}\right)\!.
\end{align*}
The vertical vector of the right hand side of (\ref{eq:matrix6}) is composed of all $\pi^{m}{\rm Ls}_{\mathbf{k}}^{\mathbf{l}}\left(\frac{\pi}{3}\right)$ satisfying $m+|\mathbf{k}|=k+1$ and $\sum_{u=1}^{n}k_{u}-1-l_{u}\equiv1\ {\rm mod}\ 2$.
 The row echelon form of matrix $A$ is
\[\left(\begin{array}{rrrrrrrrrrrrrrrr}
0 & 0 & 1 & 1 & 0 & 0 & 0 & 0 & -\frac{1}{3} & 0 & \frac{1}{36} & \frac{5}{48} & 0 & -\frac{1}{12} & \frac{25}{432} & -\frac{1}{72} \\
0 & 0 & 0 & 0 & 1 & 0 & 0 & -\frac{3}{16} & 0 & 0 & \frac{1}{4} & -\frac{7}{48} & 0 & \frac{11}{24} & -\frac{7}{48} & \frac{7}{1296} \\
0 & 0 & 0 & 0 & 0 & 1 & 0 & -\frac{45}{16} & 0 & 0 & \frac{5}{2} & -\frac{15}{8} & 0 & \frac{45}{8} & \frac{25}{8} & -\frac{25}{16} \\
0 & 0 & 0 & 0 & 0 & 0 & 1 & \frac{3}{4} & 0 & 0 & -1 & -\frac{1}{4} & 0 & -\frac{3}{4} & \frac{1}{4} & 0 \\
0 & 0 & 0 & 0 & 0 & 0 & 0 & 0 & 0 & 1 & \frac{1}{12} & \frac{5}{16} & 0 & -\frac{1}{4} & -\frac{1}{48} & \frac{5}{216} \\
0 & 0 & 0 & 0 & 0 & 0 & 0 & 0 & 0 & 0 & 0 & 0 & 1 & \frac{9}{4} & \frac{3}{2} & -\frac{3}{4}
\end{array}\right).\]
Because the eight entries of the vertical vector of the right hand side of (\ref{eq:matrix6}) from the bottom are multiples of $\pi$, we focus on rows of the row echelon form of matrix $A$ whose leading coefficient is within the eighth from the right. Then we have
\begin{align}\left(\begin{array}{r}
0 \\
0 
\end{array}\right)=\left(\begin{array}{rrrrrrrr}
0 & 1 & \frac{1}{12} & \frac{5}{16} & 0 & -\frac{1}{4} & -\frac{1}{48} & \frac{5}{216} \\
0 & 0 & 0 & 0 & 1 & \frac{9}{4} & \frac{3}{2} & -\frac{3}{4}
\end{array}\right)\left(\begin{array}{r}
{\rm Ls}_{3 , 2}^{(0 , 0)}\left(\frac{\pi}{3}\right) \\
{\rm Ls}_{2 , 3}^{(0 , 0)}\left(\frac{\pi}{3}\right) \\
{\rm Ls}_{5}^{(1)}\left(\frac{\pi}{3}\right) \\
{\rm Ls}_{5}^{(3)}\left(\frac{\pi}{3}\right) \\
\pi {\rm Ls}_{4}^{(0)}\left(\frac{\pi}{3}\right) \\
\pi {\rm Ls}_{4}^{(2)}\left(\frac{\pi}{3}\right) \\
\pi^{2} {\rm Ls}_{3}^{(1)}\left(\frac{\pi}{3}\right) \\
\pi^{3} {\rm Ls}_{2}^{(0)}\left(\frac{\pi}{3}\right)
\end{array}\right).\label{eq:im5}
\end{align}

We also calculate the imaginary parts of MZVs of weight $k-1-2m$ $(0 \le m \le (k-4)/2)$. On our example, We only need to calculate when the weight is $4$. By calculating the imaginary parts of all MZVs of weight $4$ except when it is self-dual (but when $\zeta(\mathbf{k})$ is included as an element of the vector, then we omit  $\zeta(\mathbf{k}^*)$), we have
\begin{align*}\left(\begin{array}{r}
0 
\end{array}\right)
=\left(\begin{array}{r}
\Im(\zeta\left(4\right))
\end{array}\right)
=\left(\begin{array}{rrrr}
\frac{1}{6} & \frac{3}{8} & \frac{1}{4} & -\frac{1}{8}
\end{array}\right)
\left(\begin{array}{r}
{\rm Ls}_{4}^{(0)}\left(\frac{\pi}{3}\right) \\
{\rm Ls}_{4}^{(2)}\left(\frac{\pi}{3}\right) \\
\pi{\rm Ls}_{3}^{(1)}\left(\frac{\pi}{3}\right) \\
\pi^{2}{\rm Ls}_{2}^{(0)}\left(\frac{\pi}{3}\right)
\end{array}\right).
\end{align*}
Next, we multiply relations obtained by calculating the imaginary parts of MZVs of weight $k-1-2m$ by $\pi^{1+2m}$. On our example, we have 
\begin{align}\left(\begin{array}{r}
0 
\end{array}\right)
=\left(\begin{array}{rrrr}
\frac{1}{6} & \frac{3}{8} & \frac{1}{4} & -\frac{1}{8}
\end{array}\right)
\left(\begin{array}{r}
\pi{\rm Ls}_{4}^{(0)}\left(\frac{\pi}{3}\right) \\
\pi{\rm Ls}_{4}^{(2)}\left(\frac{\pi}{3}\right) \\
\pi^{2}{\rm Ls}_{3}^{(1)}\left(\frac{\pi}{3}\right) \\
\pi^{3}{\rm Ls}_{2}^{(0)}\left(\frac{\pi}{3}\right)
\end{array}\right).\label{eq:im4}
\end{align}

By applying (\ref{eq:im5}) and (\ref{eq:im4}) to (\ref{eq:matrix5}), we obtain
\begin{align}
\left(\begin{array}{r}
\zeta\left(5\right) \\
\zeta\left(1, 4\right) \\
\zeta\left(2, 3\right) \\
\zeta\left(3, 2\right)
\end{array}\right)
=\left(\begin{array}{r@{\kern0.5em}r@{\kern0.5em}r@{\kern0.5em}r@{\kern0.5em}r@{\kern0.5em}
r@{\kern0.5em}r@{\kern0.5em}r}
0 & 0 & -\frac{1}{12} & \frac{3}{16} & 0 & -\frac{1}{4} & -\frac{1}{16} & \frac{1}{24} \\
0 & 0 & -\frac{1}{6} & \frac{3}{8} & 0 & -\frac{1}{2} & \frac{1}{8} & 0 \\
0 & 0 & \frac{11}{24} & -\frac{33}{32} & 0 & \frac{11}{8} & -\frac{13}{32} & \frac{1}{48} \\
0 & 0 & -\frac{3}{8} & \frac{27}{32} & 0 & -\frac{9}{8} & \frac{7}{32} & \frac{1}{48}
\end{array}\right)\left(\begin{array}{r}
{\rm Ls}_{3 , 2}^{(0 , 0)}\left(\frac{\pi}{3}\right) \\
{\rm Ls}_{2 , 3}^{(0 , 0)}\left(\frac{\pi}{3}\right) \\
{\rm Ls}_{5}^{(1)}\left(\frac{\pi}{3}\right) \\
{\rm Ls}_{5}^{(3)}\left(\frac{\pi}{3}\right) \\
\pi {\rm Ls}_{4}^{(0)}\left(\frac{\pi}{3}\right) \\
\pi {\rm Ls}_{4}^{(2)}\left(\frac{\pi}{3}\right) \\
\pi^{2} {\rm Ls}_{3}^{(1)}\left(\frac{\pi}{3}\right) \\
\pi^{3} {\rm Ls}_{2}^{(0)}\left(\frac{\pi}{3}\right)
\end{array}\right)\label{eq:nmatrix5}.
\end{align}
Here, $l_{k}$ is given as the rank of the first matrix of the right hand side of (\ref{eq:nmatrix5}).

Moreover, the row echelon form of 
\begin{align*}
\left(\begin{array}{rrrrrrrr|r}
0 & 0 & -\frac{1}{12} & \frac{3}{16} & 0 & -\frac{1}{4} & -\frac{1}{16} & \frac{1}{24} & \zeta(5)\\
0 & 0 & -\frac{1}{6} & \frac{3}{8} & 0 & -\frac{1}{2} & \frac{1}{8} & 0 & \zeta(1,4)\\
0 & 0 & \frac{11}{24} & -\frac{33}{32} & 0 & \frac{11}{8} & -\frac{13}{32} & \frac{1}{48} & \zeta(2,3)\\
0 & 0 & -\frac{3}{8} & \frac{27}{32} & 0 & -\frac{9}{8} & \frac{7}{32} & \frac{1}{48}& \zeta(3,2)
\end{array}\right)
\end{align*}
is
\[\left(\begin{array}{rrrrrrrr|r}
0 & 0 & 1 & -\frac{9}{4} & 0 & 3 & 0 & -\frac{1}{4} & -6 \, \zeta\left(5\right) - 3 \, \zeta\left(1, 4\right) \\
0 & 0 & 0 & 0 & 0 & 0 & 1 & -\frac{1}{3} & -8 \, \zeta\left(5\right) + 4 \, \zeta\left(1, 4\right) \\
0 & 0 & 0 & 0 & 0 & 0 & 0 & 0 & -\frac{1}{2} \, \zeta\left(5\right) + \zeta\left(2, 3\right) + 3 \, \zeta\left(1, 4\right) \\
0 & 0 & 0 & 0 & 0 & 0 & 0 & 0 & -\frac{1}{2} \, \zeta\left(5\right) + \zeta\left(3, 2\right) - 2 \, \zeta\left(1, 4\right)
\end{array}\right).\]
Threrfore, we obtain independent relations $0=-\frac{1}{2} \, \zeta\left(5\right) + \zeta\left(2, 3\right) + 3 \, \zeta\left(1, 4\right)$ and $0=-\frac{1}{2} \, \zeta\left(5\right) + \zeta\left(3, 2\right) - 2 \, \zeta\left(1, 4\right)$.
\section*{Acknowledgment}
The author is deeply grateful to Prof. Kohji Matsumoto for helpful comments. He is also deeply grateful to Prof. Iwao Kimura for some advice on computer.

\vspace{4mm}

{\footnotesize
{\sc
\noindent
Graduate School of Mathematics, Nagoya University,\\
Chikusa-ku, Nagoya 464-8602, Japan.
}\\
{\it E-mail address}, R. Umezawa\hspace{1.75mm}: {\tt
m15016w@math.nagoya-u.ac.jp}\\
}

\end{document}